\documentclass[a4paper]{amsart}
\usepackage[utf8]{inputenc}
\usepackage{amsmath,amsthm,amssymb,thmtools}
\usepackage{cleveref}
\input{grcawol.sty}
\input{grcalcx.sty}
\input{grcalck.sty}
\renewcommand\theta{\vartheta}
\numberwithin{equation}{section}
\newcommand\MC{{ \mathcal M}}% for a module category

\newtheorem{Lem}{Lemma}[section]
\newtheorem{Prop}[Lem]{Proposition}
\newtheorem{Cor}[Lem]{Corollary}
\newtheorem{Thm}[Lem]{Theorem}

\newtheorem{DefLem}[Lem]{Definition and Lemma}
\theoremstyle{definition}
\newtheorem{Def}[Lem]{Definition}
\theoremstyle{remark}
\newtheorem{Rem}[Lem]{Remark}
\newtheorem{Expl}[Lem]{Example}
\newcommand\coend{\operatorname{coend}}
\newcommand\Nat{\operatorname{Nat}}

\newcommand\ot{\otimes}
\newcommand\ou[1]{\otimes_{ {#1}}}
\DeclareMathOperator\id{\operatorname{id}}

\newcommand\ev{\operatorname{ev}}
\newcommand\db{\operatorname{db}}
\newcommand{\leer}{\operatorname{--}}
\newcommand\sw[1]{{}_{(#1)}}
\newcommand\swe[1]{ {}_{[#1]}}

\newcommand\ol{\overline}
\newcommand\inv{^{-1}}

\newcommand\LMod[1]{{_{#1}\mathfrak M}}
\newcommand\RComod[1]{{\mathfrak M^{#1}}}
\newcommand\RComodf[1]{{\mathfrak M_{\operatorname{\sf fd}}^{#1}}}
\renewcommand\epsilon\varepsilon

\newcommand\C{\mathcal C}
\newcommand\B{\mathcal B}

\newcommand\ctr{\operatorname{\mathcal Z}}

\def\BHM#1.#2.#3.#4.{{^{#1}_{#3}\mathcal B^{#2}_{#4}}}

 \makeatletter
 \newcommand\coad{ {\operatorname{coad}}}
 \newcommand\grbr{
	\grcalca = \grcolumn
	\grcalcb = \grrow
	\multiply \grcalca by \factor
	\advance \grcalca by \hfactor
	\multiply \grcalcb by \factor
	\grcalcc = \grcalcb
	\advance \grcalcc by -\hfactor
	\grcalcd = \grcalca
	\advance \grcalcd by \factor
	\put(\grcalcd,\grcalcb){\line(0,-1){\factor}}
	\put(\grcalcd,\grcalcc){\circle*{\qfactor}}
	\qbezier(\grcalca,\grcalcb)(\grcalca,\grcalcc)(\grcalcd,\grcalcc)
	\advance \grcalcb by -\factor
	\qbezier(\grcalca,\grcalcb)(\grcalca,\grcalcc)(\grcalcd,\grcalcc)
	\advance\grcolumn by 2
}

\newcommand\refc{\mathcal E}
 \title{Reflective Centers as Categories of Modules}
 \author{Peter Schauenburg}
 \address{Université Bourgogne Europe, CNRS, IMB UMR 5584, 21000 Dijon, France}
  \email{peter.schauenburg@ube.fr}
 \newcommand\oH{ {\mathbb H}}
 
\begin{document}
 \thanks{The author thanks Chelsea Walton for her comments on the first version which helped improve the second.}

 \maketitle
 \begin{abstract}
	In \cite{laugwitzwaltonyakimov} the authors construct the reflective center of a module category $\MC$ over a braided monoidal category $\B$. The reflective center is by construction a braided module category over $\B$. In the case where $\B$ is the category of modules over a finite dimensional quasitriangular Hopf algebra $H$, acting on the category of modules over a comodule algebra, they construct a comodule algebra, the reflective algebra, whose modules are precisely the reflective center. In the construction, Majid's transmutation of $H$ plays a crucial rôle. 

 This note centers on the transmuted $H$, seeking to ``explain'' its appearance through a generalization in which the acting category is no longer a module category, but admits an internal reconstructed Hopf algebra; the transmutation is a special case of this notion. As a result, in certain cases, the reflective center is simply the category of modules in $\MC$ over that Hopf algebra in $\B$.
\end{abstract}
\section{Introduction}

In \cite{laugwitzwaltonyakimov}, the authors construct a class of braided module categories over a braided monoidal category $\B$. Namely, given a module category $\MC$ over $\B$ they construct, by a sort of analog of the Drinfeld double construction, the reflective center $\refc_\B(\MC)$, which is a braided module category over $\B$. The term ``reflective center'' is reflective of the relation of braided module categories to the quantum reflection equation, parallel to the relation of braided monoidal categories to the quantum Yang-Baxter equation; we will not discuss this background and motivation at all but rather refer the reader to \cite{laugwitzwaltonyakimov} for a comprehensive review and an equally comprehensive bibliography.

As a next step, the authors of \cite{laugwitzwaltonyakimov} consider the case where $\B$ is the category of $H$-modules for a quasitriangular Hopf algebra, and $\MC$ is the category of modules over an $H$-comodule algebra $A$, the categorical action coming from the comodule algebra structure. They identify
\begin{equation}\label{LWY-DH}
 \refc_{\LMod H}(\LMod A)\cong{^{\hat H}_A{\text{\sf DH}(H)}}
\end{equation}
where the right hand side is a category of Doi-Hopf modules, in which crucially a version $\hat H$ of Majid's transmutation of $H$ features. In the case where $H$ is finite dimensional they also consider an $H$-comodule algebra $R_H(A)$, the reflective center, such that
\begin{equation}\label{LWY-Alg}
	{^{\hat H}_A{\text{\sf DH}(H)}}\cong \LMod{R_H(A)}.
\end{equation}
The appearance of the transmuted Hopf algebra in this construction struck the present author as rather mysterious, and too good to be left as a beautiful coincidence.

The aim of this note is thus to make the appearance of the transmuted Hopf algebra less of a mysterious coincidence. We do this by generalizing the results to a setting where the acting category is no longer required to be associated to a $k$-Hopf algebra. What we require instead is that it admits an internally reconstructed Hopf algebra $H$ coacting universally on all the objects in $\B$. We also require some technical conditions which will be discussed in due course, but we can already remark that Majid's transmutation of a coquasitriangular Hopf algebra fulfils them for an action, for example, on the category of modules over a module algebra, and that exact module categories for finite tensor categories also fall under our setup.

In the simplest form our result then says
\begin{equation}
 \refc_\B(\MC)\cong{_H\MC}
\end{equation}
that is, in the correct setup the reflective center is the category of modules in $\MC$ over the internal Hopf algebra in $\B$. This can be viewed as a far reaching analog of \cref{LWY-DH}.

The simple form of the result relies crucially on a special structure of the internal Hopf algebra $H$, namely the structure of a commutative central braided Hopf algebra introduced in \cite{NeuSch:RBCNCB,Sch:CBHA}. This involves a special half-braiding $\sigma$ making $H$ an algebra in the center of $\B$

In the case where the category $\MC$ is described as a category of modules within $\B$ over an algebra in $\B$, our result allow us to describe the reflective center in the same way:
\begin{equation*}
	\refc_\B(\B_A)\cong\B_{A\ou\sigma H}
\end{equation*}
with the tensor product algebra formed with the half-braiding.

As an analog of \cref{LWY-Alg} we will discuss how to ``bosonize'' the module category $_H\MC$ in the case where $H$ is transmuted from an ordinary Hopf algebra whose comodule category acts on the module category over a module algebra.

Throughout the paper we make rather free use of any notions of tensor and module category theory, hoping the reader is already well versed or will find all they need in \cite{MR3242743,MR2119143}. An exception will be made for the notions around central braided bialgebras which we assume to be less widely known, so that we will recall their axioms.

We will also make free use of the usual graphical calculus. The action of $\B$ on $\MC$ will look no different from the tensor product in $\B$ except for the special feature that only one object of $\MC$ can occur, and must stay firmly fixed to the right margin. We use
\begin{equation*}
	\gbeg21\gmu\gnl\gend,\gbeg21\gcmu\gnl\gend,\gbeg21\glm\gnl\gend,\gbeg21\grcm\gnl\gend,\gbeg21\gbr\gnl\gend,\gbeg21\gbrc\gnl\gend
\end{equation*}
for the multiplication of an algebra, the comultiplication of a coalgebra, a left module structure, a right comodule structure, the braiding of a braided category, and the half-braiding in the Drinfeld center construction, respectively. For an algebra $A$ in a monoidal category $\C$, and a $\C$-module category $\MC$, we write $_A\MC$ for the category of modules over $A$ in the category $\MC$, and similarly $^C\MC$ for comodules in $\MC$ over a coalgebra $C$ in $\C$, while the notations $\LMod A$ and $\RComod C$ are reserved for the categories of modules over a $k$-algebra, resp.\ comodules over a $k$-coalgebra.

\section{Module braidings from commutative central braided Hopf algebras}
Throughout, we will consider a braided monoidal category $\B$ and a $\B$-module category $\MC$. We will consider both the monoidal associators and the module associators as strict, backed up by the usual coherence theorems. We denote the action bifunctor by $\triangleright$.

\begin{Def}\label{def:brm}
	Let $\B$ be a braided monoidal category. A braided left module category over $\B$ is a left $\B$-module category $(\MC,\triangleright)$ equipped with an isomorphism $e_{X,V}\colon X\triangleright V\to X\triangleright V$ (the module braiding) natural in $X\in\B$, $V\in \MC$ satisfying
	$$
	e_{X\ot Y,V}=
	\gbeg36
        \got 1X\got 1Y\got 1V\gnl
	\gibr\gcl1\gnl
	\gcl1\grbr\gnl
	\gbr\gcl1\gnl
	\gcl1\grbr\gnl
	\gob 1X\gob 1Y\gob 1V\gend
	\qquad\text{and}\qquad
	e_{X,Y\triangleright V}=
	\gbeg35
	\got 1X\got 1Y\got 1V\gnl
	\gbr\gcl1\gnl
	\gcl1\grbr\gnl
	\gbr\gcl1\gnl
	\gob 1X\gob 1Y\gob 1V\gend
$$
where $e_{X,V}=\gbeg 23\got 1X\got 1V\gnl\grbr\gnl\gob 1X\gob 1V\gend$.\end{Def}

Our aim is to describe certain braided module category structures, and notably the reflective center to be recalled below, by means of module structures in $\MC$ over certain Hopf algebras in $\B$. The first question that arises would be how $_H\MC$ is a $\B$-module category. This is not a problem in presence of a braiding, but to clarify what is going on we formulate this (surely well-known) fact in the situation where the acting category is not braided.
\begin{Rem}\label{modovercentralalg}
	Let $\C$ be a monoidal category, $\MC$ a $\C$-module category, and $B$ an algebra in the center of $\C$. Then $_B\MC$ is a $\C$-module category, with the $B$-module structure on $X\triangleright M$ for $X\in\B$  and $M\in{_B\MC}$ defined by
\begin{equation*}
	\gbeg34
	\got1B\got1X\got1M\gnl
	\gbrc\gcl1\gnl
	\gcl1\glm\gnl
	\gob1X\gvac1\gob1M\gend
\end{equation*}
and the underlying functor $_B\MC\to\MC$ a strict $\C$-module functor.
\end{Rem}
We will apply this to a Hopf algebra in the braided category $\B$ acting on $\MC$, but crucially we will neither use the braiding of $\B$ nor its inverse when considering the Hopf algebra as an algebra in the center of $\B$.

Rather, we now recall the notion of a central braided bialgebra and its ``good'' comodules from \cite{ Sch:CBHA}. If the reader does not want to refer back to \cite{ Sch:CBHA}, some of the axioms may seem unmotivated and strange. We will recall in the next section how such objects arise in fact very naturally.

A \emph{central coalgebra} in $\B$ is a coalgebra $H$ with a half-braiding (potentially different from the braiding in $\B$) making it  an object in the center of $\B$  in such a way that

  \begin{equation}\label{cctreq}
    \gbeg 35
    \got 2H\got 1X\gnl
    \gcmu\gcl 1\gnl
    \gcl 1\gbrc\gnl
    \gibr\gcl 1\gnl
    \gob 1X\gob 1H\gob 1H
    \gend
    =
    \gbeg 35
    \got 1H\got 2X\gnl
    \gbbrhc3214\gnl\gnl
    \gcl 1\gcmu\gnl
    \gob 1X\gob 1H\gob 1H
    \gend
    =
    \gbeg 35
    \got 2H\got 1X\gnl
    \gcmu\gcl 1\gnl
    \gcl 1\gbr\gnl
    \gbrc\gcl 1\gnl
    \gob 1X\gob 1H\gob 1H
    \gend
  \end{equation}
  holds for all $X\in\B$. Note that this requirement is quite different from the more obvious condition that $H$ be a coalgebra in the center.

  Central braided coalgebras form a monoidal category under the tensor product of coalgebras and of the center. A \emph{central braided bialgebra} is an algebra in the category of central braided coalgebras. Thus it is an ordinary bialgebra whose underlying coalgebra is a central braided coalgebra, and such that the multiplication is a morphism in the center. An $H$-comodule $X$ is called \emph{good} if
\begin{gather}
    \label{gcv1}
    \gbeg 34
    \got 1X\gvac1\got1Y\gnl
    \grcm\gcl1\gnl
    \gcl1\gbrc\gnl
    \gob1X\gob1Y\gob1H
    \gend
    =
    \gbeg 35
    \got1X\got1Y\gnl
    \gbr\gnl
    \gcl1\grcm\gnl
    \gbr\gcl1\gnl
    \gob1X\gob1Y\gob1H
    \gend
    =
    \gbeg 36
    \got1X\gvac1\got1Y\gnl
    \grcm\gcl1\gnl
    \gcl1\gbr\gnl
    \gbr\gcl2\gnl
    \gbr\gnl
    \gob1X\gob1Y\gob1H
    \gend
    \end{gather}
    for all $Y\in\B$. We denote the category of good comodules by $\B^{!H}$. It is a monoidal subcategory of the category $\B^H$ of all $H$-comodules. If $H$ is a commutative braided central bialgebra, that is, the multiplication is commutative with respect to the half-braiding, then the braiding in $\B$ between two comodules in $\B^{!H}$ is in fact an $H$-comodule morphism, and thus $\B^{!H}$ is braided with the braiding induced by that of $\B$.

    \begin{Prop}\label{modules to reflective center} 
	    Let $\B$ be a braided monoidal category, $\MC$ a $\B$-module category, and $H$ a commutative central braided Hopf algebra in $\B$. Then the category $_H\MC$ of left $H$-modules in $\MC$ is a braided left $\B^{!H}$-module category over the category of good $H$-comodules as follows:

	The $\B^{!H}$-module structure on $_H\MC$ is pulled back from the $\B$-module structure on $_H\MC$ from \Cref{modovercentralalg} using the half-braiding of $H$.

	The module braiding is given by 
	\begin{equation}\label{braidingfrommodule}
		e_{X,M}
	=\gbeg 34
	 \got 1X\gvac 1\got1M\gnl
	 \grcm\gcl1\gnl
	 \gcl1\glm\gnl
 \gob 1X\gvac1\gob1M\gend
 \end{equation}
	 for $X\in\B^{!H}$ and $M\in{_H\MC}$.
\end{Prop}
\begin{proof}
	We need not comment on the module category structure of $_H\MC$.

	We note that the definition of $e_{XM}$ can be given for $X\in\B^H$ and $M\in {_H\MC}$ without using the good comodule condition or, for that matter, the central braided bialgebra structure.

		It is easy to check that $e_{X,M}$ is an isomorphism with inverse given by $e\inv_{X,M}=(X\ot\mu)(X\ot S\ot M)(\delta\ot M)$.

		We need to check that $e_{X,M}$ is an $H$-module morphism, which is true for any $X\in\B^H$ but uses braided commutativity of $H$:
		\begin{equation*}
			\gbeg47
			\got 1H\gvac1\got1X\got1M\gnl
			\gbbrhc3215\gcl2\gnl
			\gnl
			\grcm\glm\gnl
			\gcl2\gcn2113\gcl1\gnl
			\gvac2\glm\gnl
			\gob1X\gvac2\gob1M
			\gend
			=
			\gbeg48
			\got1H\got1X\gvac1\got1M\gnl
			\gcl1\grcm\gcl3\gnl
			\gbrc\gcl1\gnl
			\gcl4\gbrc\gnl
			\gvac1\gcl1\glm\gnl
			\gvac1\gcn2113\gcl1\gnl
			\gvac 2\glm\gnl
			\gob1X\gvac2\gob1M\gend
			=
			\gbeg48
			\got1H\got1X\gvac1\got1M\gnl
			\gcl1\grcm\gcl5\gnl
			\gbrc\gcl1\gnl
			\gcl4\gbrc\gnl
			\gvac1\gmu\gnl
			\gvac1\gcn2123\gnl
			\gvac2\glm\gnl
			\gob1X\gvac2\gob1M\gend
			=
			\gbeg47
			\got1H\got1X\gvac1\got1M\gnl
			\gcl1\grcm\gcl5\gnl
			\gbrc\gcl1\gnl
			\gcl3\gmu\gnl
			\gvac1\gcn2123\gnl
			\gvac2\glm\gnl
			\gob1X\gvac2\gob1M\gend
			=
			\gbeg47
			\got1H\got1X\gvac1\got1M\gnl
			\gcl2\grcm\gcl1\gnl
			\gvac1\gcl1\glm\gnl
			\gbrc\gvac1\gcl2\gnl
			\gcl2\gcn2113\gnl
			\gvac2\glm\gnl
			\gob1X\gvac2\gob1M
			\gend
   		\end{equation*}

	For $X,Y\in\B^H$ (and using only that $H$ is a bialgebra) we have
	\begin{equation}\label{firstaxfrommodule}
		\gbeg 55
		\gvac 1\got1{X\ot Y}\gvac 2\got 1M\gnl
		\gvac 1\grcm\gvac1\gcl1\gnl
		\gvac 1\gcl1\gcn2113\gcl1\gnl
		\gvac 1\gcl1\gvac1\glm\gnl
		\gvac 1\gob1{X\ot Y}\gvac 2\got1M\gend
		=
		\gbeg57
		\got 1X\gvac 1\got 1Y\gvac1\got 1M\gnl
		\grcm\grcm\gcl4\gnl
		\gcl4\gbr\gcl1\gnl
		\gvac 1\gcl3\gmu\gnl
		\gvac 2\gcn2123\gnl
		\gvac 3\glm\gnl
		\gob 1X\gob1Y\gvac2\gob1M
		\gend
		=
		\gbeg56
		\got 1X\gvac1\got1Y\gvac1\got1M\gnl
		\grcm\grcm\gcl1\gnl
		\gcl3\gbr\glm\gnl
		\gvac1\gcl2\gcn2113\gcl1\gnl
		\gvac 3\glm\gnl
		\gob1X\gob1Y\gvac2\gob1M
		\gend
		=
		\gbeg 46
		\got 1X\got1Y\gvac1\got1M\gnl
		\gcl1\grcm\gcl1\gnl
		\gbr\glm\gnl
		\gcl1\grcm\gcl1\gnl
		\gibr\glm\gnl
		\gob1X\gob1Y\gvac1\gob1M
		\gend
	\end{equation}
	This looks like the first axiom in the definition of a module braiding, and it is in fact this axiom if $X,Y\in\B^{!H}$ because then the braiding in the picture is also the braiding in the acting category.

	Next, for $X\in \B^{!H}$ and $Y\in\B$ we have
		\begin{equation*}
		\gbeg45
		\got1X\gvac2\got1{Y\triangleright M}\gnl
		\grcm\gvac1\gcl2\gnl
		\gcl2\gcn2113\gnl
		\gvac2\glm\gnl
		\gob1X\gvac2\gob1{Y\triangleright M}\gend
		=
		\gbeg45
		\got1X\gvac1\got1Y\got1M\gnl
		\grcm\gcl1\gcl2\gnl
		\gcl2\gbrc\gnl
		\gvac1\gcl1\glm\gnl
		\gob1X\gob1Y\gvac1\gob1M
		\gend
		=
		\gbeg45
		\got1X\got1Y\gvac1\got1M\gnl
		\gbr\gvac1\gcl2\gnl
		\gcl1\grcm\gnl
		\gbr\glm\gnl
		\gob1X\gob1Y\gvac1\gob1M
		\gend
	\end{equation*}
	which looks like the second axiom for a module braiding. In fact if $Y\in\B^{!H}$ it is this second axiom because then the braiding in the picture is the braiding in the acting category.
\end{proof}
\begin{Def}
	Let $\MC$ be a $\B$-module category over a braided monoidal category $\B$. The \emph{reflective center} of $\MC$ is the category $\refc_\B(\MC)$ whose objects are pairs $(V,e_{\leer,V})$ in which $V\in\MC$ and $e_{XV}\colon X\triangleright V\to X\triangleright V$ is an isomorphism natural in $X\in\B$ that satisfies $e_{IV}=\id_V$ as well as the first of the two equations in \Cref{def:brm}. Morphisms in $\refc_\B(\MC)$ are morphisms in $\MC$ commuting with all the $e$ morphisms in the obvious sense.
\end{Def}

In \cite{laugwitzwaltonyakimov} the authors show that $\refc_\B(\MC)$ is indeed a braided module category, where $Y\triangleright(V,e_{\leer,V})=(Y\triangleright V,e_{\leer,Y\triangleright V})$ with $e_{X,Y\triangleright V}$ \emph{defined} by the second diagram in \Cref{def:brm}.
	
\begin{Rem}\label{univprop}
	If $\MC$ is a monoidal category, braidings on $\MC$ correspond to monoidal functors $\MC\to\ctr(\MC)$ to the Drinfeld center that are sections of the underlying functor. An analogous fact holds for the reflective center. But because the acting category is actually not affected by the construction, the reflective center even has a more direct couniversal property. We formulate it a bit sloppily, a 2-categorical formulation taking isomorphisms between functors into account would be more appropriate.

	Let $\MC$ be a $\B$-module category and $\C$ a braided $\B$-module category. Then $\B$-module functors $\C\to \MC$ correspond to braided $\B$-module functors $\C\to\refc_\B(\MC)$. Namely, given a module functor $\mathcal F\colon\C\to\mathcal \MC$ with its coherent isomorphism $X\triangleright \mathcal F(M)\cong\mathcal F(X\triangleright M)$, the unique lifting $\mathcal G\colon\C\to\refc_\B(\MC)$ to a braided module functor $\mathcal G\colon \C\to\mathcal \refc_\B(\MC)$ is given by $\mathcal G(M)=(\mathcal F(M),e_{X,\mathcal F(M)})$ with $e_{X,\mathcal F(M)}=(X\triangleright \mathcal F(M)\cong\mathcal F(X\triangleright M)\xrightarrow{\mathcal F(e_{XM})}\mathcal F(X\triangleright M)\cong X\triangleright \mathcal F(M))$.
\end{Rem}
\begin{Cor}
	Let $H$ be a commutative central braided Hopf algebra in the braided monoidal category $\B$ and $\MC$ a $\B$-module category. Then the braided $\B^{!H}$-module structure of $_H\MC$ above gives rise to a $\B^{!H}$-module functor $_H\MC\to\refc_{\B^{!H}}(\MC)$.
\end{Cor}
\section{Strong coendomorphism objects}
Very roughly speaking Tannaka duality deals with reconstructing a Hopf algebra from a monoidal category and a monoidal functor to the category of vector spaces.
The idea to use a general braided category as a target category instead, and in particular to reconstruct a Hopf algebra within the category itself from its identity endofunctor goes back to \cite{MR1188817}.

The following definition makes sense if $\B$ is just monoidal but we will assume $\B$ braided throughout.
\begin{Def}
	Let $\mathcal X$ be a category and $\omega\colon\mathcal X\to\B$ a functor. A \emph{coendomorphism object} of $\omega$ in $\B$ is an object $\coend(\omega)\in\B$ with a natural transformation $\delta\colon \omega\to\omega\ot \coend(\omega)$ which is universal in that the natural transformation
	\begin{align*} 
		\B(\coend(\omega),T)&\to\Nat(\omega,\omega\ot T)\\
		f&\mapsto (\omega\ot f)\delta
	\end{align*}
	is an isomorphism for all $T\in\B$.

	A coendomorphism object of $\B$ is a coendomorphism object of the identity endofunctor of $\B$.
\end{Def}
It follows from the definition that $\coend(\omega)$ is a coalgebra in $\B$ in such a way that the components of $\delta$ are comodule structures.

It is (pardon the swipe) well-known but false that if $\mathcal X$ is monoidal, $\omega$ is a monoidal functor, and $\coend(\omega)$ is a coendomorphism object, then $\coend(\omega)$ is a bialgebra. It is true in many interesting cases however, and the following stronger definition was proposed in \cite{ Sch:CBHA} as a unifying way to capture  frequently encountered conditions that ``make it work'':
\begin{Def}
	A coendomorphism object of $\omega\colon\mathcal X\to \B$ is \emph{strong} if
	\begin{align*}
		\B(\coend(\omega)\ot P,T)&\to\Nat(\omega\ot P,\omega\ot T)\\
	f&\mapsto (\omega\ot f)(\delta\ot P)
\end{align*}
is a bijection for all $P,T\in\B$.
\end{Def}

In light of our present interests we could advance the following definition, which recovers the preceding one by considering the $\B$-module category $\mathcal B$.
\begin{Def}
	Let $\MC$ be a $\B$-module category. A coendomorphism object of $\omega\colon\mathcal X\to \B$ is $\MC$-\emph{strong} if
	\begin{align*}
		\MC(\coend(\omega)\triangleright P,T)&\to\Nat(\omega\triangleright P,\omega\triangleright T)
	\end{align*}
	is a bijection for all $P,T\in\MC$.
\end{Def}
The following consequence of $\MC$-strongness is what, in the case $\MC=\B$, makes it true that the strong coendomorphism object of a monoidal functor is a bialgebra. Namely, the way in which the definition ``fixes'' the technical problem alluded to above is that if two functors to $\B$ admit strong coendomorphism objects, so does their tensor product (defined on the product of the source categories). In this way the use of ``strong'' replaces the assumption made for example in \cite{MR1381692} that coendomorphism objects are compatible in this way with tensor products.
\begin{Lem}
	Let $\omega\colon \mathcal X\to \B$ and $\nu\colon\mathcal Y\to \B$ admit strong $\MC$-strong coendomorphism objects $(\coend(\omega),\delta)$, $(\coend(\nu),\delta')$ Then $\omega\ot\nu\colon\mathcal X\times\mathcal Y\to\B$ admits the strong and $\MC$-strong coendomorphism object
	\begin{equation*}
		\omega\ot\nu\xrightarrow{\delta\ot\delta'}\omega\ot\coend(\omega)\ot\nu\ot\coend(\nu)\xrightarrow{\id\ot c\ot\id}\omega\ot\nu\ot\coend(\omega)\ot\coend(\omega),
	\end{equation*}
	where $c$ is the braiding of $\B$.
\end{Lem}
\begin{proof}We use the notation $\Nat_X(\omega(X),\nu(X)):=\Nat(\omega,\nu)$ if we feel the need to refer to the object ($X$) that a transformation between two functors is natural in. For bifunctors and transformations natural in two objects we use $\Nat_{X,Y}$. Now we have
	\begin{align*}
		\MC((\coend(\omega)\ot\coend(\nu))\triangleright P,T)
		&\cong\MC(\coend(\omega)\triangleright\coend(\nu)\triangleright P,T)\\
		&=\Nat_X(\omega(X)\triangleright \coend(\nu)\triangleright P,\omega(X)\triangleright T)\\
		&=\Nat_X((\omega(X)\ot\coend(\nu))\triangleright P,\omega(X)\triangleright T)\\
		&=\Nat_X((\coend(\nu)\ot\omega(X))\triangleright P,\omega(X))\triangleright T)\\
		&=\Nat_X(\coend(\nu)\triangleright\omega(X)\triangleright P,\omega(X)\triangleright T)\\
		&=\Nat_{X,Y}(\nu(Y)\triangleright\omega(X)\triangleright P,\nu(Y)\triangleright \omega(X)\triangleright T)\\
		&=\Nat((\omega\ot\nu)\triangleright P,(\omega\ot\nu)\triangleright T).
	\end{align*}
\end{proof}

We believe that coendomorphism objects do not always exist, but do so in very many interesting cases, that they are not always strong but they are in many interesting cases, and that they are not always $\MC$-strong but they are in many interesting cases. 

We'll discuss two cases of interest where the technical condition of $\MC$-strongness is met.

We first point out that if $X\in\B$ has a left dual, then we have an adjunction
\begin{equation*}
	\MC(X^*\triangleright M,N)\cong\MC(M,X\triangleright N)
\end{equation*}
for every $\B$-module category $\MC$, which is proved in exactly the same way as for $\MC=\B$.
\begin{Lem}\label{strongness lemma}
	Let $\B$ be braided monoidal, $\MC$ a $\B$-module category, and $\omega\colon\mathcal X\to\B$ a functor. Assume that
	\begin{enumerate}
		\item Every functor $\B\ni X\mapsto X\triangleright M\in\MC$ has a right adjoint $\MC\ni N\mapsto\hom(M,N)\in\B$.
		\item Every $\omega(X)$ has a left dual.
	\end{enumerate}
	Then every coendomorphism object of $\omega$ is $\MC$-strong.
\end{Lem}
\begin{proof}
 We have $Y\ot\hom(M,N)\cong\hom(M,Y\triangleright N)$ since
	\begin{align*}
		\B(X,Y\ot(\hom(M,N)))
		&\cong \B(Y^*\ot X,\hom(M,N))\\
		&\cong \MC((Y^*\ot X)\triangleright M,N)\\
		&\cong \MC(Y^*\triangleright X\triangleright M,N)\\
		&\cong \MC(X\triangleright M,Y\triangleright N)\\
		&\cong \B(X,\hom(M,Y\triangleright N))
	\end{align*}
and therefore
	\begin{align*}
		\Nat(\omega\triangleright M,\omega\triangleright N)
		&\cong\Nat(\omega,\hom(M,\omega\triangleright N))\\
		&\cong\Nat(\omega,\omega\ot\hom(M,N))\\
		&\cong\B(\coend(\omega),\hom(M,N))\\
		&\cong\MC(\coend(\omega)\triangleright M,N)
	\end{align*}
\end{proof}

\begin{Expl}
	Let $\B$ be a braided finite tensor category. Then $\B$ has a coendomorphism object which is $\MC$-strong for every exact $\B$-module category.
\end{Expl}
In fact the coendomorphism object is well known and goes back to work of Lyubashenko \cite{MR1324034}. The right adjoint $\hom$ is known from \cite{MR2119143}
\begin{Expl}
	Let $\oH$ be a coquasitriangular Hopf algebra over a field $k$ and consider the category $\B=\RComod\oH$ of right $\oH$-comodules. The category $\B$ admits a coendomorphism object $H$ which is Majid's transmuted Hopf algebra
	%\cite[Section 9.4]{MR1381692}. 
	It is strong and it is $\MC$-strong for any module category $\MC$ equivalent to the category of modules over a $k$-algebra $A$, provided the action functor $\triangleright$ is a $k$-bilinear functor exact in each argument.
\end{Expl}
Majid's transmutation appears first (for the dual case of quasitriangular Hopf algebras) in \cite{MR1129171}. We refer to \cite[Section 9.4]{MR1381692}, notably Example 9.4.10. We will need more details on the transmuted Hopf algebra $H$ associated to $\oH$ below, here we only discuss why it has the relevant properties. By definition $H$ is the coendomorphism object of (the identity functor on) $\B=\RComod\oH$. By the finiteness theorem for comodules it is also the coendomorphism object of the inclusion functor $\omega\colon\mathcal X\to\B$ where $\mathcal X$ is the category of finite-dimensional $\oH$-comodules. It is not hard to see that for each $M\in\MC$ the functor $\RComod H\to {\LMod A}$ given by $X\mapsto X\triangleright M$ has a right adjoint, and therefore $H$ is a strong coendomorphism object and $\MC$-strong as claimed. In fact it should be true that every exact functor $\RComod C\to\LMod A$ for a coalgebra $C$ and algebra $A$ has a right adjoint, but I could not pinpoint a reference for this.

Now let $\B$ be a braided monoidal category that admits a strong coendomorphism object $H$. By definition there is a functor $\B\to \B^H$ defined by the universal natural (in $X\in\B$) transformation $X\to X\ot H$. There is a unique structure of commutative central braided bialgebra on $H$ for which the functor has its image in $\B^{!H}$. We'll assume that $H$ is a Hopf algebra (which is the case for example if $H$ is Majid's transmutation of a coquasitriangular Hopf algebra, or if $H$ is the coendomorphism object of a finite tensor category). 
\begin{Thm}
	Let $\B$ be a braided monoidal category admitting a strong coendomorphism object $H$ which is a Hopf algebra. Let $\MC$ be a $\B$-module category for which $H$ is $\MC$-strong. Then
	\begin{equation*}
		\refc_\B(\MC)\cong {_H\MC}.
	\end{equation*}
\end{Thm}
\begin{proof}
	We have seen in \Cref{modules to reflective center} that $_H\MC$ is a braided $\B^{!H}$-module category, and so by pullback it is a braided $\B$-module category. By \Cref{univprop} this defines a functor $_H\MC\to\refc_\B(\MC)$. The functor is an equivalence, since by $\MC$-strongness, a natural in $X\in \B$ transformation
	$e_{XM}\colon X\triangleright M\to X\triangleright M$ is given by a morphism $\mu \colon H\triangleright M\to M$. In fact if we depict $\mu$ in the shape of a module structure (before proving it is one), the isomorphism
	\[\Nat_X(X\triangleright M,X\triangleright M)\cong \MC(H\triangleright M,M)\] is given by the same picture as \cref{braidingfrommodule}.

	To see that the morphism $\mu$ thus corresponding to a module braiding defines a module structure on $M$ we run \cref{firstaxfrommodule} backwards:
		\begin{equation*}
			\gbeg57
		\got 1X\gvac 1\got 1Y\gvac1\got 1M\gnl
		\grcm\grcm\gcl4\gnl
		\gcl4\gbr\gcl1\gnl
		\gvac 1\gcl3\gmu\gnl
		\gvac 2\gcn2123\gnl
		\gvac 3\glm\gnl
		\gob 1X\gob1Y\gvac2\gob1M
		\gend
		=
		\gbeg 55
		\gvac 1\got1{X\ot Y}\gvac 2\got 1M\gnl
		\gvac 1\grcm\gvac1\gcl1\gnl
		\gvac 1\gcl1\gcn2113\gcl1\gnl
		\gvac 1\gcl1\gvac1\glm\gnl
		\gvac 1\gob1{X\ot Y}\gvac 2\got1M\gend
			=
		\gbeg 46
		\got 1X\got1Y\gvac1\got1M\gnl
		\gcl1\grcm\gcl1\gnl
		\gbr\glm\gnl
		\gcl1\grcm\gcl1\gnl
		\gibr\glm\gnl
		\gob1X\gob1Y\gvac1\gob1M
		\gend
=
		\gbeg56
		\got 1X\gvac1\got1Y\gvac1\got1M\gnl
		\grcm\grcm\gcl1\gnl
		\gcl3\gbr\glm\gnl
		\gvac1\gcl2\gcn2113\gcl1\gnl
		\gvac 3\glm\gnl
		\gob1X\gob1Y\gvac2\gob1M
		\gend
		\end{equation*}
		This equality of two natural automorphisms of $(X\ot Y)\triangleright M$ shows the equality of two morphisms $H\ot H\ot M\to M$, namely the module associativity for $M$.

		This shows that every module braiding on $M\in\MC$ arises from a unique $H$-module structure on $M$.
\end{proof}
\begin{Rem}\label{isoremark}
	Note that in the proof the property of a module braiding to be an isomorphism wasn't actually used, only the two equations it has to satisfy. As a consequence, we have proved that under the hypotheses of the Theorem, every natural transformation satisfying the two equations for a module braiding is in fact a natural isomorphism, the inverse given by the antipode. 

	Of course the existence of an antipode is closely related to the existence of dual objects, and in fact one can prove that if a rigid braided monoidal category $\B$ acts on a category $\MC$, then every natural transformation satisfying the equations of a module braiding is necessarily an isomorphism; this is similar to the behavior of the Drinfeld center construction on a rigid monoidal category. In fact, let $V^*$ be a left dual of $V\in\B$ with evaluation $\ev\colon V^*\ot V\to I$ depicted $\gbeg21\gev\gend$ and coevaluation $\db\colon I\to V\ot V^*$ depicted $\gbeg21\gdb\gend$. We have (writing $\ol c=c\inv$)
	\begin{align*}
		\gbeg33
		\got1{V^*}\got1V\got1M\gnl
		\gev\gcl1\gnl
		\gvac2\gob1M\gend
		=
		\gbeg35
		\got1{V^*}\got1V\got1M\gnl
		\gbr\gcl3\gnl
		\gibr\gnl
		\gev\gnl
		\gvac2\gob1M\gend
		=
		\gbeg57
		\got3{V^*\ot V}\gvac1\got1M\gnl
		\gvac1\gbmp c\gvac2\gcl2\gnl
		\gcn4137\gnl
		\gvac3\grbr\gnl
		\gvac3\gbmp{\ol c}\gcl2\gnl
		\gvac3\gbmpt\ev\gnl
		\gvac4\gob1M\gend
		=
		\gbeg39
		\got1{V^*}\got1V\got1M\gnl
		\gbr\gcl2\gnl
		\gibr\gnl
		\gcl1\grbr\gnl
		\gbr\gcl1\gnl
		\gcl1\grbr\gnl
		\gibr\gcl2\gnl
		\gev\gnl
		\gvac2\gob1M
		\gend
		=
		\gbeg37
		\got1{V^*}\got1V\got1M\gnl
		\gcl1\grbr\gnl
		\gbr\gcl1\gnl
		\gcl1\grbr\gnl
		\gibr\gcl2\gnl
		\gev\gnl
		\gvac2\gob1M\gend,
	\end{align*}
	proving that
	\[
		\gbeg47
		\gvac2\got1V\got1M\gnl
		\gdb\gcl1\gcl2\gnl
		\gcl4\gbr\gnl
		\gvac1\gcl1\grbr\gnl
		\gvac1\gibr\gcl2\gnl
		\gvac1\gev\gnl
		\gob1V\gvac2\gob1M\gend
	\]
	is a left inverse for $e_{V,M}$. The proof that $e_{V,M}$ has a right inverse is similar.
\end{Rem}

An interesting class of $\B$-module categories is given by the categories $\B_A$ of right $A$-modules in $\B$, where $A$ is an algebra in $\B$. In fact, all exact module categories over a finite tensor category are of this form \cite[Theorem 3.17]{MR2119143}. We note that a strong coendomorphism object of $\B$ is also $\B_A$-strong 

For a $\B$-module category of the form $\B_A$ the following result describes, in particular, the reflective center in the same form. We prefix the result by an obvious general remark which is surely folklore.
\begin{Rem}
	Let $\C$ be a monoidal category, $A$ an algebra in $\C$ and $B$ an algebra in $\ctr(\C)$ with half-braiding $\sigma$. Then $A\ot_\sigma B=A\ot B$ is an algebra in $\C$ with multiplication 
	\begin{equation*}
	\gbeg 44
	\got 1A\got1B\got1A\got1B\gnl
	\gcl1\gbrc\gcl1\gnl
	\gmu\gmu\gnl
	\gob2A\gob2B\gend
\end{equation*}
Right $A\ot_\sigma B$-modules $M$ are the same as right $A$-modules and $B$-modules satisfying
\begin{equation}\label{rightmodulecomm}
	\gbeg35
	\got1M\got1B\got1A\gnl
	\grm\gcn121{-1}\gnl
	\gcl1\gnl
	\grm\gnl
	\gob1M\gend
	=
	\gbeg35
	\got1M\got1B\got1A\gnl
	\gcl1\gbrc\gnl
	\grm\gcn111{-1}\gnl
	\grm\gnl
	\gob1M\gend
\end{equation}
and algebra morphisms $A\ot_\sigma B\to R$ for an algebra $R$ are in bijection with pairs of algebra morphisms $f\colon A\to R$ and $g\colon B\to R$ satisfying
\begin{equation}\label{faux-coproduct}
	\gbeg25
	\got1B\got1A\gnl
	\gcl1\gcl1\gnl
	\gbmp g\gbmp f\gnl
	\gmu\gnl
	\gob2R\gend
	=
	\gbeg25
	\got1B\got1A\gnl
	\gbrc\gnl
	\gbmp f\gbmp g\gnl
	\gmu\gnl
	\gob2R\gend
\end{equation}
\end{Rem}
\begin{Cor}
	Let $A$ be an algebra in the braided monoidal category $\B$ and $H$ a strong coendomorphism object in $\B$.
	\begin{enumerate}
		\item Braidings making $\B_A$ a braided module category over $\B$ are parametrized by algebra maps $f\colon H\to A$ satisfying 
			\begin{equation}\label{centercond}
				\gbeg25
				\got1H\got1A\gnl
				\gcl1\gcl2\gnl
				\gbmp f\gnl
				\gmu\gnl
				\gob2A\gend
				=
				\gbeg25
				\got1H\got1A\gnl
				\gbrc\gnl
				\gcl1\gbmp f\gnl
				\gmu\gnl
				\gob2A\gend
			\end{equation}
		\item $\refc_\B(\B_A)\cong \B_{A\ot_{\sigma}H}$ where $\sigma$ is the half-braiding of the central braided bialgebra $H$.
	\end{enumerate}
\end{Cor}
\begin{proof}
	We already know $\refc_{\B}(\B_A)\cong{_H\B_A}$. Since $H$ is a commutative algebra in the center, left modules correspond to right modules via the braiding, and the bimodule associativity turns into \cref{rightmodulecomm} under the correspondence because composing the right hand side with $\sigma_{HM}$ gives
	\begin{equation*}
			\gbeg36
	\got1H\got1M\got1A\gnl
	\gbrc\gcl1\gnl
	\gcl1\gbrc\gnl
	\grm\gcn111{-1}\gnl
	\glm\gnl
	\gob3M\gend
	=
	\gbeg35
	\got1H\got1M\got1A\gnl
	\gcl1\grm\gnl
	\gbrc\gnl
	\grm\gnl
	\gob1M\gend
	\end{equation*}
	This deals with the second assertion. For the first, we know by \cref{univprop} that braidings on $\B_A$ are parametrized by $\B$-module functor liftings of the underlying functor $\B_{A\ot_\sigma H}\to\B_A$. These liftings in turn are parametrized by algebra maps $A\ot_\sigma H\to A$ extending the identity on $A$, which leads to \cref{centercond} as a special case of \cref{faux-coproduct}.
\end{proof}

\section{Bosonizations}

A Hopf algebra $H$ in the braided monoidal category of $\oH$-modules over a quasitriangular Hopf algebra $\oH$  can be ``bosonized'' to give an ordinary Hopf algebra. In fact one can view $H$ as a Hopf algebra in the category of Yetter-Drinfeld modules over $\oH$, and then the bosonization is the Radford biproduct, as an algebra simply the smash product.

A similar procedure, analogous to the passage from the reflective center in a special situation via Doi-Hopf modules to the reflective algebra in \cite{laugwitzwaltonyakimov}, will lead us from modules over an internal Hopf algebra $H$ in a category acting on the category of modules over a $k$-algebra $A$ to a ``bosonized'' version of $H$.

\begin{Lem}\label{bosoring}
	Let $A$ be a $k$-algebra and assume that the category $\LMod A$ of left $A$-modules is a $\C$-category over a monoidal category $\C$ such that the action bifunctor is right exact in its right argument. If $R$ is an algebra in $\C$ then $R[A]:=R\triangleright A$ has an $A$-ring structure such that $\LMod{R[A]}\cong{_R\left(\LMod A\right)}$.
\end{Lem}
\begin{proof}
	By Watts' Theorem 
	$R\triangleright M\cong R[A]\ou A M$. 
	$R[A]$ is an $A$-ring with 
	\begin{equation*}R[A]\ou AR[A]=R\triangleright(R\triangleright A)\cong (R\ot R)\triangleright A\xrightarrow{\nabla\triangleright A}R\triangleright A\end{equation*}
	$A$-module maps $R\rightarrow M\to M$ (and among them the module structures) are the same as $A$-module maps $R[A]\ou AM\to M$ (and among them the module structures).
\end{proof}

We will consider a more concrete situation yet. Namely, let $\oH$ be a $k$-Hopf algebra. When $A$ is a right $\oH$-module algebra, then $\RComod \oH$ acts on $\LMod A$ by setting $V\triangleright M=V\ot M$ with the $A$-module structure $a(v\ot m)=v\sw 0\ot (a.v\sw 1)m$ for $a\in A,v\in V, m\in M$.

We note that any $\RComod \oH$-module category structure on $\LMod A$ for which the underlying functor to the category of vector spaces is a strict module functor can be recovered in this way from a module algebra structure on $A$. Namely, the $\oH$-action on $A$ is given in terms of the $A$-module structure on $\oH\triangleright A=\oH\ot A$ by $a.h=(\epsilon\ot A)(a(h\triangleright 1)).$ We note that by \cite{MR2331768}, if $\oH$ is finite dimensional, every exact $\RComod\oH$-module category is of this form for a suitable algebra $A$.

We also already note how module braidings are described in this situation. Namely, any natural isomorphism $e\colon V\triangleright M\to V\triangleright M$ is given by a convolution invertible linear map $K\colon H\to A$ in the form 
\begin{equation*}
	e_{VM}(v\ot m)=v\sw 0\ot K(v\sw 1)m
\end{equation*}
Of course the axioms for $e$ to be a module braiding correspond to equations for $K$.

To see what condition  ensures that the components of $e$ are $A$-module maps, we compare
\begin{align*}
	ae(v\ot m)&=a(v\sw 0\ot K(v\sw 1)m=v\sw 0\ot (a.v\sw 1)K(v\sw 2)m\\
		e(a(v\ot m))&=e(v\sw 0\ot (a.v\sw 1)m=v\sw 0\ot K(v\sw 1)(a.v\sw 2)m.
		\end{align*}
		
To find the remaining axioms we substitute the expression for $e$ in terms of $K$ into the defining equations of a module braiding:
\begin{align*}
	(V\triangleright e_{WM})&(c_{WV}\triangleright M)(W\triangleright e_{VM})(c_{VW}\triangleright M)(v\ot w\ot m)\\
	&=(V\triangleright e_{WM})(c_{WV}\triangleright M)(W\triangleright e_{VM})(w\sw 0\ot v\sw 0\ot r(v\sw 1|w\sw 1)m)\\
	&=(V\triangleright e_{WM})(c_{WV}\triangleright M)(w\sw 0\ot v\sw 0\ot K(v\sw 1)r(v\sw 2|w\sw 1)m\\
		&=(V\triangleright e_{WM})(v\sw 0\ot w\sw 0\ot r(w\sw 1|v\sw 1)K(v\sw 2)r(v\sw 3|w\sw 2)m\\
		&=v\sw 0\ot w\sw 0\ot K(w\sw 1)r(w\sw 2|v\sw 1)K(v\sw 2)r(v\sw 3|w\sw 3)m
	\end{align*}
	
	wants to equal
	\begin{align*}
		e_{V\ot W,M}(v\ot w\ot m)=v\sw 0\ot w\sw 0\ot K(v\sw 1w\sw 1)m
	\end{align*}
	while
	\begin{align*}
		(c_{WV}\triangleright M)&(W\triangleright e_{VM})(c_{VW}\triangleright M)(v\ot w\ot m)\\
		&=(c_{WV}\triangleright M)(W\triangleright e_{VM}(w\sw 0\ot v\sw 0\ot r(v\sw 1|w\sw 1)m)\\
		&=(c_{WV}\triangleright M)(w\sw 0 v\sw 0\ot K(v\sw 1)r(v\sw 2\ot w\sw 2)m\\
		&=v\sw 0\ot w\sw 0\ot r(w\sw 1|v\sw 1)K(v\sw 2)r(v\sw 3|w\sw 2)
	\end{align*}
	
	wants to equal
	\begin{align*}	e_{V,W\triangleright M}(v\ot w\ot m)=v\sw 0\ot K(v\sw 1)(w\ot m)=v\sw 0\ot w\sw 0\ot K(v\sw 1)w\sw 1m
	\end{align*}
	
	for all $V,W\in\RComod\oH$, $M\in \LMod A$, $v\in V$, $w\in W$ and $m\in M$. Thus
	
	\begin{DefLem}
		Let $\oH$ be a coquasitriangular Hopf algebra and $A$ a right $\oH$-module algebra. Then module braidings on the $\RComod \oH$-module category $\LMod A$ correspond to \emph{reflective structures} on $A$, which are defined to be convolution invertible maps $K\colon \oH\to A$ satisfying
		\begin{align*}
			K(h\sw 1)(a.h\sw 2)&=(a.h\sw 1)K(h\sw 2),\\
			K(gh)&=K(h\sw 1)r(h\sw 2|g\sw 1)K(g\sw 2)r(g\sw 3|h\sw 3),\\			K(g).h&=r(h\sw 1|g\sw 1)K(g\sw 2)r(g\sw 3|h\sw 2).
		\end{align*}
	\end{DefLem}
	\begin{Rem}
		The map $K$ is an analog of the notion of a universal $K$-matrix introduced in \cite{MR3905136}; there specific comodule algebras rather than module algebras are considered. So $K$ could well be baptized a ``universal $K$-map''.

		Using \Cref{isoremark} the condition that $K$ be convolution invertible is actually redundant.
	\end{Rem}
	An algebra $R$ in  $\RComod \oH$ is an $\oH$-comodule algebra, and the algebra $R[A]=R\# A$ from \Cref{bosoring} is a generalized smash product, i.e. $R\# A=R\ot A$ as right module, $a(r\#x)=r\sw 0\ot (a.r\sw 1)x$ and multiplication determined by the $A$-ring condition and $(r\# 1)(s\# 1)=rs\# 1$.

\newcommand\centralhit{\leftharpoonup}
The algebra $R$ we are interested in is the transmutation $H$ of a coquasitriangular $k$-Hopf algebra $\oH$ with its dual R-matrix $r\colon \oH\ot\oH\to k$. Applying \Cref{bosoring} to $H$ we will obtain below a slight generalization (our $\oH$ need not be finite dimensional) of the construction of the reflective algebra in \cite{laugwitzwaltonyakimov}, with a suitable universal $K$-map replacing their universal $K$-matrix. To construct it explicitly we need to recall the details of Majid's transmutation of $\oH$ with all its structures. 

The braiding $c$ on the category $\RComod\oH$ is given by
\begin{equation*}
	c_{VW}(v\ot w)=w\sw 0\ot v\sw 0r(v\sw 1|w\sw 1)	.
\end{equation*}
The structure of $H$ as a Hopf algebra in $\RComod\oH$ can be found in \cite{MR1381692}. As a coalgebra $H=\oH$, and the $H$-comodule structure on an object of $\RComod\oH$ is identical to the $\oH$-comodule structure. As an $\oH$-comodule, $H=\oH^{\coad}$, that is, the $\oH$-comodule structure on $H$ is 
\begin{equation*}
	\delta\colon H\to H\ot\oH;h\mapsto h\swe0\ot h\swe1:=h\sw 2\ot S(h\sw 1)h\sw 3
\end{equation*}
Multiplication in $H$ is given by  $g\bullet h=g\sw 2h\sw 3r(g\sw 3|S(h\sw 1))r(g\sw 1|h\sw 2)$.  Finally the half-braiding making $H$ a central braided bialgebra is
determined in \cite{NeuSch:RBCNCB} to be
\begin{equation*}
	\sigma\colon H\ot V\to V\ot H;h\ot v\mapsto v\sw 0\ot h\sw 2r(v\sw 1|h\sw 1)r(h\sw 3|v\sw 2)
\end{equation*}
We rewrite it in the form $\sigma(h\ot v)=v\sw 0\ot h\centralhit v\sw 1$ with the action of $\oH$ on $H$ defined by
\begin{equation*}
	H\ot\oH\to H;h\ot g\mapsto h\centralhit g:=h\sw 2r(g\sw 1|h\sw 1)r(h\sw 2|g\sw 2).
\end{equation*}
which therefore, together with the coadjoint coaction, makes $H$ a Yetter-Drinfeld module.
\begin{Prop}
	Let $\oH$ be a coquasitriangular $k$-Hopf algebra over the field $k$, and $A$ right $\oH$-module algebra. Then $\refc_{(\RComod \oH)}(\LMod A)\cong\LMod{A[H]}$. Explicitly $A[H]=\oH\# A$ with multiplication 
	\begin{equation*}
		(g\# x)(h\# y)=g\bullet h\swe0\# (x.h\swe1)y.
	\end{equation*}
	The $\RComod \oH$-module category structure on $\LMod{A[H]}$ is given by the right $\oH$-module algebra structure
	\begin{equation*}
		(g\# x).h=g\centralhit h\sw 1\# x.h\sw 2.
	\end{equation*}
	The reflective structure $K\colon\oH\to A[H]$ is given by $K(h)=h\#1$.
\end{Prop}
\begin{proof}
	We already know that the multiplication is the right one. The $\RComod\oH$ module structure on $\LMod{A[H]}={_H\left(\LMod A\right)}$ is the one described in \cref{modovercentralalg}. That is to say, for $M\in\LMod{A[H]}$ and $V\in\RComod\oH$ we have $V\triangleright M=V\ot M$ with the $A[H]$-module structure given by the previous $A$-module structure and 
	\begin{equation*}
		h.(v\ot m)=v\sw 0\ot (h\centralhit v\sw 1)m.
	\end{equation*}
	so that 
	\begin{align*}
		(h\# x)(v\ot m)&=h.(v\sw 0\ot(x.v\sw 1)m)\\
		&=v\sw 0\ot (h\centralhit v\sw 1)(x.v\sw 1)m
        \end{align*}
	In particular, the $A[H]$-module structure on $H\triangleright A[H]$ is given by
	\begin{align*}
		(g\# x)(h\ot j\# y)=h\sw1\ot (g\centralhit h\sw2\#x)(j\# y)
	\end{align*}			
	Thus it corresponds to the module algebra structure
	\begin{align*}
		(g\# x).h&=(\epsilon\ot \oH\ot A)( (g\ot x).(h\ot 1\# 1))\\
		&=(\epsilon\ot\oH\ot A)(h\sw 1\ot g\centralhit h\sw 2\# x.h\sw 3\\
			&=g\centralhit h\sw 1\# x.h\sw 2.
	\end{align*}
	The module braiding on $\LMod{A[H]}$ is given by
	\begin{align*}
		e_{VM}(v\ot m)=v\sw 0\ot v\sw1m
	\end{align*}
	which implies the form of $K$.
\end{proof}
\section{quasi-appendix}
We close with a sketch of how the preceding construction should generalize, without \emph{substantial} additional pains, to the case of a coquasitriangular coquasi-Hopf algebra. We will not actually carry this through to the point of presenting an explicit result. We think it worthwhile, however, to discuss how the construction above is sufficiently well set up to take care of much of the technical difficulties that usually make explicit calculations with quasi-Hopf algebras or their dual counterparts rather painful. In fact we will argue our point without even needing to reference the explicit axioms of the objects involved, evoking their conceptual meaning will suffice.

By  a coquasitriangular coquasi-Hopf algebra we mean a $k$-coalgebra $\oH$ equipped with a multiplication admitting a unit element, and a convolution invertible form $\phi\colon \oH^{\ot 3}\to k$ such that the category $\RComod\oH$ is monoidal with the tensor product induced by multiplication, the unit object corresponding to the unit element, and the associativity constraint
\[U\ot (V\ot W)\ni u\ot v\ot w\mapsto u\sw 0\ot v\sw 0\ot w\sw 0\phi(u\sw 1|v\sw 1|w\sw 1)\in (U\ot V)\ot W.\]
Further $\oH$ should be equipped with a coalgebra anti-automorphism and two maps $\alpha,\beta\colon \oH\to k$ such that the category $\RComodf\oH$ of finite dimensional comodules is rigid in the following way: For $V\in\RComodf\oH$ the dual space $V^*$ with the comodule structure determined by $\langle\kappa\sw 0|v\rangle\kappa\sw 1=\langle\kappa|v\sw 0\rangle S(v\sw 1)$ for $v\in V$ and $\kappa\in V^*$ is a dual object of $V$ in $\RComodf\oH$ with evaluation and coevaluation
\begin{gather*}
	\ev\colon V^*\ot V\ni\kappa\ot v\mapsto \langle\kappa|v\sw 0\rangle\alpha(v\sw 1)\in k\\
	\db\colon k\to (v_i)\sw 0\beta((v_i)\sw 0)v^i\in V\ot V^*.
\end{gather*}
Finally $\oH$ carries a convolution invertible $r\colon \oH\ot\oH\to k$ such that the same formula used above for an ordinary Hopf algebra equips $\RComod \oH$ with a braiding.

Of course there are equations involving multiplication, $\phi$, $S$, $\alpha$, $\beta$, and $r$ equivalent to the statement that $\RComodf\oH$ is a rigid braided monoidal category in the indicated fashion, but we will not work to a degree of detail making their knowledge necessary.

Conceptually, it is clear how transmutation theory for a coquasitriangular coquasi-Hopf algebra should work. The case of a quasitriangular quasi-Hopf algebra was considered in detail in \cite{2009arXiv0903.3959K} (except for a lack of interest in the central braided Hopf algebra structure). We do not know if the dual situation, which is better suited for the infinite dimensional case, has been looked at anywhere. For the sequel we conclude that $\RComod \oH$ contains a strong coendomorphism object $H$ which is a commutative central braided Hopf algebra.

To describe a suitable action of $\RComod\oH$ on the category of modules $\LMod A$ over an algebra $A$ we first note that given a map $A\ot\oH\to A$ denoted $a\ot h\mapsto a.h$ and satisfying $(ab).h=(a.h\sw 1)(b.h\sw 2)$ as well as $1.h=\epsilon(h)1$, we can define a functor $\triangleright\colon\RComod\oH\times\LMod A\to\LMod A$ by endowing $V\triangleright M$ for $V\in\RComod\oH$ and $M\in\LMod A$ with an $A$-module structure exactly as above in the case of an ordinary Hopf algebra. To turn $\triangleright$ into a module category structure we need to describe the module associator. Now any natural transformation $\Psi\colon V\triangleright W\triangleright M\to(V\ot W)\triangleright M$
in $V,W\in \RComod \oH$ and $M\in\LMod A$ has the form $\Psi(v\ot w\ot m)= v\sw 0\ot w\sw 0\ot \psi(v\sw 1|w\sw 1)m$ for some map $\psi\colon \oH\ot\oH\to A$. It is not a problem (only complicated) to translate the conditions that $\Psi$ be an $A$-module map and fulfil the coherence condition for a module category into axioms for $\psi$. Collecting these we would arrive at the suitable definition of an $\oH$-module algebra structure for $A$ (note this is not \emph{that} straightforward since $\oH$ is not associative). Similar notions with $\oH$ but not $A$ dualized are well-studied, see \cite{MR3929714} and its bibliography. 

Likewise, one can study module braidings on the module category associated to such a module algebra $A$. They would be parametrized like before by some linear map $K\colon \oH\to A$. Again, it would not be a problem, but again complicated, to extract the relevant conditions on $K$. One has to run the tentative definition of the module braiding (which does not change with respect to the ordinary Hopf case) through the axioms in \cref{def:brm}, this time inserting module associators for all the necessary shifting of parentheses that is implicit in the pictures.

As for the construction of the reflective algebra $A[H]$ describing the reflective center it would work exactly the same, except that the result would be peppered generously with instances of $\phi,\psi,\alpha$ and $\beta$. Notably, in spite of all the appearances of explicit nontrivial associators, there would be no need to prove that the obtained multiplication on $A[H]\cong  \oH\# A$ is in fact associative and fulfils the axioms of a module algebra. As for associativity, it is taken care of by the fact that the use of Watts' theorem to define the algebra structure of  $H\triangleright A$ automatically turns the ostensibly nontrivial associativities into the simple associativity of the tensor product of bimodules over a ring. The principle of this trick and some applications were extensively discussed in \cite{Sch:AMCGHSP}. The only work to be done is reading off the multiplication, module structure, and reflective structure of $A[H]$ already knowing that they are present, much like done above in the ordinary Hopf case.
\bibliographystyle{alpha}
\bibliography{eigene,andere,arxiv,mathscinet}

\end{document}